\documentclass[12pt]{amsart}
\usepackage{amssymb,amsmath,amsthm,amsopn,amsfonts,graphics,xy,epsfig,picture,epic,version}
\usepackage[dvipsnames]{xcolor}
\usepackage{a4wide}

\def\ffi{\varphi}
\def\eps{\varepsilon}
\def\dst{\displaystyle}

\renewcommand{\Im}{\mathrm{Im}\,}
\renewcommand{\Re}{\mathrm{Re}\,}

\DeclareMathOperator{\dist}{dist}
\def\C{{\mathbb{C}}}

\def\N{{\mathbb{N}}}

\def\R{{\mathbb{R}}}

\def\Z{{\mathbb{Z}}}

\def\d{\,{\mathrm{d}}}

\newcommand{\norm}[1]{{\left\|{#1}\right\|}}
\newcommand{\ent}[1]{{\left[{#1}\right]}}
\newcommand{\abs}[1]{{\left|{#1}\right|}}

\newcommand{\nc}{\newcommand}
\nc{\lra}{\longrightarrow}

\newtheorem{lemma}{Lemma}[section]
\newtheorem{proposition}[lemma]{Proposition}
\newtheorem{theorem}[lemma]{Theorem}
\newtheorem{corollary}[lemma]{Corollary}

\theoremstyle{definition}
\newtheorem{definition}[lemma]{Definition}

\theoremstyle{remark}
\newtheorem{remark}[lemma]{Remark}

\title{Quantitative estimates of sampling constants in model spaces}

\author{A. Hartmann, P. Jaming \& K. Kellay}

\subjclass[2010]{41A10;42C15 }
\keywords{Model space,  Bernstein inequalities, sampling, reverse Carleson measure}
\thanks{The authors kindly acknowledge financial support from the French ANR program, ANR-12-BS01-0001 (Aventures),
the Austrian-French AMADEUS project 35598VB - Char\-ge\-Disq, the French-Tunisian CMCU/UTIQUE project 32701UB Popart, the Joint French-Russian Research Project PRC CNRS/RFBR 2017-2019.\\
This study has been carried out with financial support from the French State, managed
by the French National Research Agency (ANR) in the frame of 
the Investments for
the Future Program IdEx Bordeaux - CPU (ANR-10-IDEX-03-02).  }

\address{Address: Univ. Bordeaux, IMB, UMR 5251, F-33400 Talence, France.
CNRS, IMB, UMR 5251, F-33400 Talence, France.}  
\email{Andreas.Hartmann@math.u-bordeaux.fr}
\email{Philippe.Jaming@math.u-bordeaux.fr}
\email{Karim.Kellay@math.u-bordeaux.fr}

\begin{document}

\begin{abstract} 
We establish quantitative estimates for sampling (dominating) 
sets in model spaces
associated with meromorphic inner functions, {\it i.e.\ }those 
corresponding to de Branges spaces.
Our results encompass the Logvinenko-Sereda-Panejah (LSP) Theorem including Kovrijkine's
optimal sampling constants for Paley-Wiener spaces. 
It also extends Dyakonov's LSP theorem for model spaces associated with bounded
derivative inner functions.
Considering meromorphic inner functions allows us to introduce a new geometric density condition,
in terms of which the sampling sets are completely characterized.
This, in comparison to Volberg's characterization of sampling measures in terms of harmonic measure, 
enables us to obtain explicit estimates on the sampling constants.
The methods combine Baranov-Bernstein inequalities, reverse Carleson measures and
Remez inequalities.
\end{abstract}

\subjclass[2000]{46E22,42C15,30J05}

\maketitle

\section{Introduction}

An important question in signal theory is to know how much information
of a signal is needed in order to be recovered exactly. This information can be given 
in discrete form (points giving rise to so-called sampling sequences) 
or more generally by subsets of
the sets on which the signal is defined. A prominent class of signals is given by the Paley-Wiener
space $PW^2_{\sigma}$ of entire functions of type $\sigma>0$ and which are square 
integrable on the
real line (finite energy). By the Paley-Wiener theorem it is known that this is exactly the
space of Fourier transforms of functions in $L^2(-\sigma,\sigma)$ (finite energy signals
on $(-\sigma,\sigma)$). For the Paley-Wiener space, it is known (Shannon-Whittaker-Kotelnikov
Theorem) that $\dst\frac{\pi}{\sigma}\Z$ forms a sampling sequence for $PW^2_{\sigma}$, which means
that each function can be exactly resconstructed from its values on $\dst\frac{\pi}{\sigma}\Z$ with an
appropriate control of norm (actually every sequence $\gamma\Z$ with $\gamma\le\pi/\sigma$ is
sampling for $PW_{\sigma}^2$). 
More general sequences can be considered; we refer to the seminal paper \cite{HNP} 
which provides criteria, yet difficult to check, involving the famous Muckenhoupt condition.

Panejah \cite{Pa1,Pa2}, Kac'nelson \cite{Kac} and Logvinenko-Sereda \cite{LS}
were interested in a characterization of subsets $\Gamma\subset \R$ allowing
to recover Paley-Wiener functions, or more precisely their norm, 
from their restriction to $\Gamma$. In other words, they were looking for sets
$\Gamma$ for which there exists a constant $C=C(\Gamma)$ such that, for every $f\in PW_{\sigma}^2$,
\begin{equation}\label{eq:DefSamp}
\int_\R|f(x)|^2\d x\leq C\int_\Gamma|f(x)|^2\d x.
\end{equation}
Such sets $\Gamma$  are said to be {\em dominating} and the least constant appearing in 
\eqref{eq:DefSamp} is called the {\em sampling constant} of $\Gamma$. 
In view of the result on sampling sequences, it is not too
surprising that in order to be dominating in the Paley-Wiener space,
$\Gamma$ has to satisfy a so-called relative density condition. This
condition means that each interval $I$ of a given length $a$ ($a$ large enough) contains a
minimal proportion of $\Gamma$, 
$$
 |\Gamma\cap I|\geq \gamma|I|
$$ 
for some $\gamma>0$ independent of $I$.
A very natural question is to establish a link between $\gamma$ and $C(\Gamma)$. This is particularly interesting for applications where one wants to measure $f$ on a 
set $\Gamma$ as small as possible while aiming at an estimate closest possible 
to the norm of $f$. This requires
a knowledge of the sampling constants depending on the size of $\Gamma$. For the Paley-Wiener
space, essentially optimal quantitative estimates of these constants were given by Kovrijkine \cite{Ko}, {see also} \cite{Re}. 

The Paley-Wiener space is a special occurrence of so-called model spaces (see below for precise
definitions), which do not only occur in the setting of signal theory, but also in control problems
as well as in the
context of second order differential operators (e.g. Schr\"odinger operators, Sturm-Liouville problems, which naturally connect the theory to that of de Branges spaces of entire functions 
which are unitarily equivalent to a subclass of model spaces), {see} \cite{MP} for an interesting account of such connections. As it turns out the
problem of dominating sets in model spaces was completely solved by Volberg who provided
a description for general inner functions in terms of harmonic measure \cite{V}
(see also \cite{HJ1} for $p=1$).
Dyakonov was more interested in a geometric
description for dominating sets \cite{Dy}. He revealed that the notion of relative density indeed 
generalizes to a much broader class than just Paley-Wiener spaces. His achievement is that
relatively dense sets are dominating in a model space $K_{\Theta}^p$ precisely when the inner function
$\Theta$ defining the model space has bounded derivative. 

The aim and novelty of this paper is to consider an appropriate density condition
and quantitive estimates of the sampling
constants in the broader setting of model spaces associated to general meromorphic inner functions
(these correspond exactly to the situation of de Branges spaces of entire functions). 
Our framework is larger than Dyakonov's in that we allow $\Theta$ to have an unbounded
derivative, and this immediately leads to the question on how to measure the size of the set $\Gamma$. 
Indeed, according to Dyakonov's results,
relative density is intimately related to the boundedness of $\Theta'$, and so, in our
more general setting, we have to replace uniform intervals of length $a$, which were already considered
in the Paley-Wiener space in \cite{Pa1,Pa2,Kac,LS} or more generally in \cite{Dy} when
$\Theta'$ is bounded, by a suitable family of intervals. The main tool in this direction is
a Whitney type covering of $\R$ introduced by Baranov \cite{B1} in which the length of the
test intervals is given in terms of the distance to some
level set of $\Theta$ (as a matter of fact, in the Paley-Wiener space this distance is constant).
Once we obtain a geometric characterization of dominating sets, we determine the dependence of
the sampling constant $C(\Gamma)$ on the parameters arising in the characterization.

Let us discuss some of the main ingredients used in this paper.
A central tool is a Remez-type inequality which allows to estimate mean
values of holomorphic functions on an interval in terms of their means 
on a measurable subset of the interval. This involves some uniform estimates of the function
which we will explain a little bit more below.
It is essentially this Remez-type inequality which determines the
dependence of the sampling constants on the density.
The Remez-inequality requires also some relative
smallness condition of the intervals we have to consider. For this reason we
first need to reduce the problem to intervals which satisfy this smallness condition. 
For this reduction we introduce a class of test sets which have a fixed size with respect to the
Baranov intervals (independently of $\Gamma$). The sampling constants for this class of test sets
turn out to be 
uniformly bounded from below by a reverse Carleson measure result discussed in
\cite{BFGHR}. 
In the next step, following the line of proof performed in \cite{Ko}
({\it see also} \cite{GJ}), we will show that there are sufficiently many intervals in a test set containing
points where the behavior of the functions and their derivatives are controlled by a generalized 
Bernstein type inequality originally due to Baranov. With this control, we can estimate the Taylor
coefficients of the function whose norm we are interested in, and which allows us to obtain the
uniform control required in the Remez-type inequality we alluded to above.
Finally, as in Kovrijkine's work (\cite{Ko}; similar estimates have also been used in \cite{NSV}), 
we are then able to apply the Remez type inequality on the small intervals involving the
proportion of $\Gamma$ contained in the test sets to obtain the quantitative control we are interested
in. 

We will also recall the proof of Baranov's Bernstein result with the necessary details
since we will need a slightly more precise form than that given in \cite{B2}.
\medskip

The paper is organized as follows. In the next section we introduce the necessary notation
and our main result. In the succeding section we will prove the necessity of our density condition.
Bernstein type inequalities as considered by Baranov are one main ingredient in the proof of the
sufficiency and will be
discussed in Section \ref{background}, as well as a reverse Carleson measure result in 
model spaces. In the last section we prove the sufficiency of our density condition.
\medskip

Throughout the paper, we use the notation $C(x_1,\ldots,x_n)$  to denote 
constants that depend only on some parameters $x_1,\ldots,x_n$ that
may be numbers, functions or sets. Constants may change from line to line. 

We occasionally write $A\simeq B$ to say that there exists a constant $C$ independent of $A,B$ such that
$C^{-1} A\leq B\leq CA$.

\section{Notation and statement of the main result}

In order to state our main result, we first need to introduce the necessary notation.
Let $\Theta$ be an inner function in the upper half--plane {$\C^+=\{z\in\C\,:\Im z>0\}$}, {\it i.e.} a bounded
analytic function with non tangential limits of modulus 1 almost everywhere on $\R$. 
For $1\leq p\leq +\infty$,  we denote by  $K_\Theta^p$ the shift co-invariant subspace 
(with respect to the adjoint of the multiplication semi-group $e^{isx}$, $s>0$)
in the Hardy class $H^p=H^p(\C^+)$,
which is given (on $\R$) by
\[
K_\Theta^p=H^p\cap \Theta \overline{H^p}.
\]
If $\Theta (z)=\Theta_\tau(z)= \exp(2i\tau z)$ for some $\tau>0$, then  $K_{\Theta_\tau}^p$ is up to the entire factor $e^{-i\tau z}$ equal to the Paley--Wiener space $PW^p_\tau$, 
which is the space of entire functions on $\C$ of exponential type at most $\tau$, whose restriction to the real line belong to $L^p(\R)$.

Before considering more general model spaces, 
let us briefly recall the situation in the Paley-Wiener space.

\begin{definition}
A measurable set $\Gamma\subset \R$ is called

--- \emph{relatively dense} if there exists $\gamma,\ell>0$ such that, for every $x\in\R$,
\begin{equation}\label{rel-dense}
|\Gamma\cap [x,x+\ell)|\geq \gamma\ell.
\end{equation}

--- \emph{dominating} for $K_\Theta^p$ if there exists a constant $C(\Theta,p,\Gamma)$
such that, for every $f\in K_\Theta^p$,
\begin{equation}
\label{eq:dominating}
\int_{\R}|f|^p\d x\le C(\Theta,p,\Gamma) \int_{\Gamma} |f|^p\d x.
\end{equation}
The smallest constant $C(\Theta,p,\Gamma)$ appearing in \eqref{eq:dominating} will be
called the sampling constant of $\Gamma$ in $K^p_{\Theta}$.
\end{definition}

Fixing $\tau$ and $p$, the
Logvinenko-Sereda theorem \cite{LS,HJ1} on equivalence of norms asserts that $\Gamma$ 
is relatively dense if and only if it is a dominating set for $PW^p_\tau$.

As mentioned in the introduction, the Logvinenko-Sereda theorem has been extended to model spaces by Volberg \cite{V} and Havin-J\"oricke \cite{HJ1}. 
Dyakonov \cite[Theorem 3]{Dy} proved that the classes
of dominating sets for $K_\Theta^p$ and relatively dense sets coincide  if and only if $\Theta$ has a bounded derivative: $\Theta'\in L^\infty(\R)$.  Though Dyakonov was not interested in explicit
constants $C(\Theta,p,\Gamma)$ a precise analysis of his method allows to obtain some estimates
in this more restrictive situation. In order to state this estimate, let us
recall the definition of  harmonic measure. For $z=x+iy\in \C^+$ and $t\in\R$, let
\[
 P_z(t)=\frac{1}{\pi}\frac{y}{(x-t)^2+y^2}
\]
be the usual Poisson kernel in the upper half plane. 
For a measurable set  $\Gamma\subset \R$, we denote by $\omega_z(\Gamma)$ 
its harmonic measure at $z$ defined by:
$$
\omega_{z}(\Gamma)=\int_{\Gamma} P_z(t)\,\mathrm{d}t.
$$
It is easily shown that $\Gamma$ is relatively dense if and only if
$\delta_y:=\inf\{\omega_z(\Gamma)\,:\ \Im z= y\}>0$ for some (all) $y>0$.

Dyakonov proved that, when $\Theta'\in L^{\infty}(\R)$ and $\delta_y>0$ then $\Gamma$ is
dominating for $K_\Theta^p$.
An estimate of the sampling constants is not given in
\cite{Dy} but may be deduced from the proof (see Remark \ref{rem1}). 

The aim of this work is to improve the estimates 
of Dyakonov's theorem as well as to establish a link between an appropriate density and the sampling constant
for general meromorphic inner functions. Meromorphic inner functions are inner functions the zeros of 
which only accumulate at infinity, and whose singular inner part is reduced to $e^{i\tau z}$, $\tau\ge 0$, {\it i.e.},
\[
 \Theta(z)=e^{i\tau z}\prod_{\lambda\in \Lambda}b_{\lambda}, \quad z\in \C_+,
\]
where $\Lambda=\{\lambda\} \subset\C_+$ is a Blaschke sequence in the upper
half plane,
$$
\sum_{\lambda\in \Lambda }\frac{\Im \lambda}{1+|\lambda|^2}<+\infty
$$
only accumulating at $\infty$. Recall that the Blaschke factor in the upper half plane
is given by
\[
 b_{\lambda}(z)=\frac{|{\lambda}^2+1|}{\lambda^2+1}\frac{z-\lambda}{z-\overline{\lambda}},\qquad z\in \C_+.
\]

In the more general situation when $\Theta$ has not necessarily bounded derivative, 
we have to adapt the concept of relative density. More precisely the size of 
the testing intervals, which was constant in the setting of classical relative density, has to take into
account the distribution of zeros of $\Theta$. To this end, we will need the notion of
sublevel set:
given $\varepsilon\in(0,1)$, this is defined by
$$
L(\Theta,\eps)=\{z\in\C_+\; :\; |\Theta(z)|<\eps\}.
$$

With this definition in mind, we can introduce one key tool in our setting. In \cite[Lemma 3.3]{B2}, 
Baranov constructed a Whitney type covering of $\R$. 
For this, 
define 
$$
d_{\eps}(x)=\dist\bigl(x,L(\Theta,\eps)\bigr).
$$
Then Baranov's construction yields a disjoint covering of $\R$
by intervals $I_n$ the length of which is comparable to the distance to the sublevel set.
More precisely, following Baranov, we have $I_n=[s_n,s_{n+1})$, where $(s_n)_n$ is a strictly
increasing sequence $\lim_{n\to\pm\infty}s_n=\pm\infty$, given by
\begin{eqnarray}\label{defBara}
 \int_{s_n}^{s_{n+1}} \frac{1}{d_{\eps}(x)}\d x=c,
\end{eqnarray}
where $c>0$ is some fixed constant.
Moreover there exists $\alpha\ge 1$ such that
\begin{eqnarray}\label{BarConst}
 \frac{1}{\alpha} d_{\eps}(x)\le |I_n|\le \alpha d_{\eps}(x),\quad
 x\in I_n.
\end{eqnarray} 
Such a sequence will be henceforth called a Baranov sequence.
In order to put our work in some more perspective to Dyakonov's work we shall recall 
an important connection between $d_{\eps}(x)$ and $|\Theta'(x)|$. This requires the
notion of the spectrum of $\Theta$ which is defined as
\[
 \sigma(\Theta)=\{z\in \C_+\cup\R\cup\{\infty\}:\liminf_{\zeta\to z}|\Theta(\zeta)|=0\}.
\]
For meromorphic inner functions $\sigma(\Theta)$ consists of the zeros of $\Theta$ and, provided
$\Theta$ is not a finite Blaschke product, the point $\infty$. 
Setting $d_0(x)=\dist(x,\sigma(\Theta))$,
Baranov showed in \cite[Theorem 4.9]{B1} that 
\[
 d_{\eps}(x)\simeq \min (d_0(x),|\Theta'(x)|^{-1}),\quad x\in\R.
\]
In particular, the Baranov intervals will be small when $\Theta'$ is big. 
\\

Recall that Volberg \cite{V} characterized dominating sets in $K_\Theta^p$ as those sets for which
$$
\inf_{z\in\C^+}\bigl(|\Theta(z)|+\omega_z(\Gamma)\bigr)>0.
$$
This characterization, based on harmonic measure, gives us an intuition that
we cannot expect to measure the size of $\Gamma$ only by looking at how much mass it puts on
a Baranov interval. Indeed, harmonic measure of a set is not very sensitive with respect to the
exact place where we put the set. For this reason we need to consider {\it amplified} intervals.
For an interval $I$ and $a>0$, we will denote by $I^a$ the amplified interval of $I$
having same center as $I$ and length $a|I|$.
\\

We are now in a position to introduce our new notion of relative density.

\begin{definition}
Let $(I_n)_{n\in\Z}$ be a Baranov sequence, $\gamma\in (0,1)$ and $a\geq 1$. A Borel set $\Gamma$ 
is called {\em $(\gamma,a)$--relatively dense} with respect to $(I_n)_{n\in\Z}$ if, for every $n\in\Z$,
\begin{eqnarray}\label{denscond}
 |\Gamma\cap I^a_n|\ge\gamma |I_n^a|=\gamma a|I_n|.
\end{eqnarray}
In case $a=1$ we will simply call the sequence $\gamma$-dense with respect to $(I_n)_{n\in\Z}$.
\end{definition}

We would like to mention that if $(I_n)_{n\in\Z}$ and $(\widetilde{I}_n)_{n\in\Z}$ are two Baranov sequences then if $\Gamma$ is {\em $(\gamma,a)$--relatively dense} with respect to $(I_n)_{n\in\Z}$  then there is a $\widetilde{\gamma}>0$
and a $\widetilde{a}>1$ such that $\Gamma$ is also {\em $(\widetilde{\gamma},\widetilde{a})$--relatively dense} with respect to 
$(\widetilde{I}_n)_{n\in\Z}$. 
This follows essentially from the fact that the Baranov intervals are defined by 
$\dst\int_{I_n}d_{\eps}^{-1}(x)\d x=c$ where $c$ is a constant (different for $(I_n)_{n\in\Z}$ and $(\widetilde{I}_n)_{n\in\Z}$), see e.g.\ \cite[Lemma 3.3]{B2}, and that neighboring intervals are of comparable length.
Therefore, in the remaining part of the paper, $(I_n)_{n\in\Z}$ will be a {\em fixed} Baranov sequence.

\smallskip

The main result of this paper is the following.

\begin{theorem}\label{thm2}
Let $p\in(1,\infty)$, $\Theta$ be a  meromorphic inner function,
and let $\Gamma\subset\R$ be a measurable set. Then the following  conditions are equivalent 
\begin{enumerate}
\renewcommand{\theenumi}{\roman{enumi}}
\item\label{condLS1} $\Gamma$ is $(\gamma,a)$--relatively dense with respect to $(I_n)_{n\in\Z}$, 
for some $\gamma>0$ and some $a\ge 1$,
\item  \label{condVolb}
$
\inf_{z\in \C^+} (\omega_z(\Gamma)+|\Theta(z)|)>0
$
\item\label{condLS2} there exists $C>0$ such that  for every $f\in K_\Theta^p$,
$$
\int_{\R}|f(x)|^p\d x\leq C \int_{\Gamma}|f(x)|^p\d x.
$$
\end{enumerate}
Moreover,  if the equivalent conditions (i)-(iii) hold, then
$$
{C}\leq\exp \Big( C(\Theta,p,\varepsilon) \frac{a^2}{\gamma}\ln \frac{1}{\gamma}\Big).
$$
\end{theorem}

We need to make some remarks here. First, the equivalence between (ii) and (iii) is of course
Volberg's result. His result works in a much broader situation since he considers arbitrary
inner functions and not only meromorphic inner ones. Also he considers not only harmonic
measure of subsets of $\R$ but harmonic extensions of general $L^1$ functions (with respect
to the measure $(1+|x|^2)^{-1}\d x$).
One novelty here is the connection
with the new notion of relative density. Next we want to discuss two main differences 
with Dyakonov's work besides the control of the sampling constant. 
The obvious difference is that the boundedness assumption of $\Theta'$
is not required. To handle that situation, the Baranov intervals will be small where $\Theta'$ is
big, and so the
$(\gamma,a)$-relative density means, losely speaking, that the 
more zeros of $\Theta$ we put somewhere, the more
mass of $\Gamma$ has to be localized there in order to correctly measure functions in $K_{\Theta}^p$.
Another observation is that even when $\Theta'$ is bounded, this result gives some new
information. More precisely, when $\Theta$ has very few zeros in certain regions, 
then the corresponding Baranov intervals will be
big, and so we distribute mass of $\Gamma$ --- comparably to the length of the Baranov interval --- wherever we want in such an interval (this is perfectly 
coherent with Volberg's harmonic measure characterization), while
classical relative density requires some uniform distribution in such big intervals.
Indeed, each subinterval of sufficiently large but fixed length has to contain a fixed portion of $\Gamma$.

A special situation occurs when $a=1$, i.e.\ when the sequence $\Gamma$ is $\gamma$-relatively 
dense with respect to $(I_n)_n$,
in other words each Baranov interval (without amplification) contains a least
proportion $\gamma$ of $\Gamma$. In this situation we improve significantly the constant. More precisely
we have the following result.
\begin{corollary}\label{Cora=1}
Let $p\in(1,\infty)$, $\Theta$ be a  meromorphic inner function,
and let $\Gamma\subset\R$ be a measurable set. If $\Gamma$ is $\gamma$-relatively dense with
respect to $(I_n)_{n\in\Z}$, $\gamma>0$, then for every $f\in K_\Theta^p$,
$$
\int_{\R}|f(x)|^p\d x\leq \Big(\frac{1}{\gamma}\Big)^{C(\Theta,p,\varepsilon)} \int_{\Gamma}|f(x)|^p\d x.
$$
\end{corollary}
We should point out that in this situation we thus obtain a polynomial dependence on 
$1/\gamma$ in accordance with Kovrijkine's optimal result in the Paley-Wiener space.

\begin{remark}\label{rem1}
Let us also compare this result to what may be obtained from Dyakonov's proof when 
$\Theta'\in L^{\infty}(\R)$. In this situation, there is a suitable $y>0$  such that
$L(\Theta,\eps)\subset \{z\in \C:\Im z>y\}$ 
(see e.g.\ \cite[p.\ 2222]{Dy}). Fix such an $y>0$ and let
$m_y=\inf\{\Theta(z)\,:\ 0< \Im z< y\}>0$ and
$\delta_y=\inf\{\omega_z(\Gamma)\,:\ \Im z= y\}$. Recall that Dyakonov proved that
if $\delta_y>0$, then $\Gamma$ is dominating. A careful inspection of his proof
further leads to the estimate
$$
\|f\|_{L^p(\R)} \leq \frac{2^\frac{1}{p\delta_y}}{m_y}\|f\|_{L^p(\Gamma)}, \qquad f\in K_\Theta^p,\ 0<p<\infty.
$$
This estimate gives an exponential control of the sampling constant depending on 
$\delta_y$ (and hence on $\gamma$), while our estimate is polynomial in $\gamma$.
\end{remark}

\begin{remark}\label{rem2}
We shall discuss here another natural guess for a necessary density condition.
As mentioned above, the Baranov intervals are given by \eqref{defBara} where the constant
$c$ is fixed arbitrarily.
Then one could think that if $\Gamma$ is dominating then there exists a suitable 
$c$ such that the associated
intervals satisfy $|I_n\cap\Gamma|\ge\gamma |I_n|$.  As it turns out, this does not work.

Here is an example :
let $\Lambda=(2^ni)_{n\ge 0}$. Then $L(\Theta,\varepsilon)$ is like a Stolz 
type angle 
$$ \Gamma_{\alpha}(0)=
\{z=x+iy: y\ge \alpha (1+|x|)\},
$$ 
and hence $d_{\varepsilon}(x)\simeq 1+|x|$.

From Volberg's characterization, it is clear that $\Gamma=\R_-$ is dominating
(in the Stolz angle $\Gamma_{\alpha}(0)$, 
the harmonic measure of $\R_-$ is bounded from below by a strictly
positive constant). However it is not dominating. Indeed from $d_{\eps}(x)\simeq 1+|x|$
it can be deduced that one can choose
$I_n=[q^n,q^{n+1})$ for $n\in\N$, and any fixed $q>1$, $I_{-n}=-I_n$.
Clearly for $n$ big enough we have $|I_n\cap \R_-|=\emptyset$.

Observe that for $a>q$ there is $\gamma>0$
such that  $|I_n^a\cap \R_-|\ge\gamma |I_n^a|$, so that $\R_-$ is
$(\gamma,a)$-dense, while it is never $\gamma$-dense with respect to any Baranov sequence.
\end{remark}

\begin{remark}
Meromorphic inner functions are those appearing in the context of de Branges spaces of 
entire functions \cite{dBR}. 
Note that for an entire function $E$ satisfying $|E(z)| \geqslant |E(\overline{z})|$, 
$\Im z > 0$, and having zeros only in the open lower half plane, the de Branges space is
defined by $\mathcal{H}(E)
=\{F \in \operatorname{Hol}(\C): {F}/{E}, {F^{*}}/{E} \in {H}^2\}$ and
$\|F\|_{\mathcal{H}(E)}=\|F/E\|_{L^2(\R)}$ (here $F^*(z)=\overline{F(\overline{z})}$). 
Since $F\mapsto F/E$ maps unitarily $\mathcal{H}(E)$ 
onto the model space ${K}_{\Theta}$, where $\Theta(z)=E^*(z)/E(z)$,  
an immediate consequence of our results is a
characterization of those measurable $\Gamma$ for which $\|F\|_{\mathcal{H}(E)}^2\le C\int_{\Gamma}
|f(x)|^2/|E(x)|^2dx$ (with the same control of constants as in the corresponding model spaces).
\end{remark}

\section{Proof of (iii) implies (i) in Theorem \ref{thm2}.}

Our aim is to test the sampling inequality on normalized reproducing kernels. 
Recall that the reproducing kernel for $K_{\Theta}^p$ at $\lambda\in \C_+$ is defined by
\begin{eqnarray}\label{repker}
 k_{\lambda}(z)=\frac{i}{2\pi}\frac{1-\overline{\Theta(\lambda)}\Theta(z)}{z-\overline{\lambda}},
 \quad z\in \C_+.
\end{eqnarray}
This means that for every $f\in K_{\Theta}^p$ and every $\lambda\in \C_+$,
we have $f(\lambda)=\langle f,k_{\lambda}\rangle
=\int_{\R} f(x)\overline{k_{\lambda}(x)}\d x$.
Then, classical estimates give for $\lambda=x+iy\in L(\Theta,\eps)$,
\[
  \|k_{\lambda} \|_p^p =\frac{1}{(2\pi)^p} \int_{\R}
 \left|\frac{1-\overline{\Theta(\lambda)}\Theta(t)}{t-\overline{\lambda}}\right|^p\d t
 \ge \frac{(1-\eps)^p}{(2\pi)^p}\int _{\R}
 \frac{1}{|t-{\lambda}|^p}\d t
 \simeq \frac{1}{y^{p-1}}  \int_{\R}\frac{1}{1+|t|^p}\d t
\]
Since $p>1$, the integral appearing in the last expression converges. Let $c_p$ be the
constant such that $\|k_{\lambda} \|_p^p \ge c_p/y^{p-1}$.
For the discussions to come we set
\[
 C_1=\left(\frac{1+\eps}{2\pi c_p}\right)^p,
\]
where $\eps$ defines the sublevel set $L(\Theta,\eps)$.

Now we are now in a position to prove the above mentioned implication.
Suppose that we have 
(iii) of Theorem \ref{thm2}
\begin{eqnarray*}
\int_{\R}|f(x)|^pdx\le C\int_{\Gamma}|f(x)|^p\d x,
\end{eqnarray*}
which we will now apply to normalized reproducing kernels. This means that 
\begin{eqnarray}\label{sameq}
 \int_{\Gamma} \left|\frac{k_{\lambda}(x)}{\|k_{\lambda}\|_p}\right|^p\d x\ge\frac{1}{C},
 \quad \lambda\in \C_+.
\end{eqnarray}

Recall from \eqref{BarConst} that there is $\alpha$ such that
\begin{eqnarray*}
 \frac{1}{\alpha} d_{\eps}(x)\le |I_n|\le \alpha d_{\eps}(x),\quad
 x\in I_n,\ n\in\N.
\end{eqnarray*} 
This is in particular true when $x=t_n$ is the center of $I_n$.
We deduce that there exists $\lambda_n=x_n+iy_n\in L(\Theta,\eps)$ with
$|t_n-\lambda_n|/\alpha \le |I_n|\le \alpha |t_n-\lambda_n|$. It should be noted that
$\Re \lambda_n$ does not need to be in $I_n$.
Still we know that
$|x_{n}-t_{n}|\le \alpha |I_{n}|$.
Then 
\[
 \tilde{I}_n:=[x_n-\frac{b}{\alpha}|I_n|,x_n+\frac{b}{\alpha}|I_n|)
 \subset [x_n-b|I_n|,x_n+b|I_n|)\subset I_n^a
\]
for $b=a/2-\alpha$ (which, in order for $b>0$ requires $a>2\alpha$).

Then given any $a\ge 1$,
\begin{eqnarray}\label{eq2}
\int_{\Gamma} \left|\frac{k_{\lambda_{n}}(x)}{\|k_{\lambda_{n}}\|_p}\right|^p\d x
 &\le& \Big(\frac{1+\eps}{2\pi c_p}\Big)^p\int_{\Gamma} 
 \frac{(\Im \lambda_{n})^{p-1}}{|x-\lambda_{n}|^{p}} \d x
 \nonumber\\
 &\le&C_1\int_{\Gamma\cap I_{n}^{a}} 
 \frac{y_{n}^{p-1}}{|x-\lambda_{n}|^{p}} \d x
 +C_1\int_{\R\setminus  I_{n}^{a}} 
 \frac{y_{n}^{p-1}}{|x-\lambda_{n}|^{p}} \d x \nonumber\\
 &\le& C_1\int_{\Gamma\cap I_{n}^{a}} 
 \frac{y_{n}^{p-1}}{|x-\lambda_{n}|^{p}} \d x
 +C_1\int_{\R\setminus  \tilde{I}_{n}} 
 \frac{y_{n}^{p-1}}{|x-\lambda_{n}|^{p}} \d x.
\end{eqnarray}
We start estimating the second integral in \eqref{eq2}.
In order to do so, we make the change of variable
$u=\big(|x-x_{n}|/y_{n}\big)^{p}$, so that
\[
 \d u=\frac{p|x-x_{n}|^{p-1}}{y_{n}^{p}}\d x.
\]
With this change of variable, and equivalence of $\ell^p$ and $\ell^2$-norms in a 
two-dimensional vector space, we get
\begin{eqnarray*}
\int_{\R\setminus  \tilde{I}_{n}} \frac{y_{n}^{p-1}}
  {|x-\lambda_{n}|^p} \d x
 &\le& 2\int_{\R\setminus  \tilde{I}_{n}} \frac{y_{n}^{p-1}}{|x-x_{n}|^p+
 y_{n_{\delta}}^p} \d x
 =\frac{2}{p}\int_{|u|\ge (b/\alpha)^p}\frac{1}{1+|u|}\frac{\d u}{|u|^{1-1/p}}\\
 &\le&\frac{4}{p}\int_{u\ge (b/\alpha)^p}\frac{\d u}{u^{2-1/p}}\\
 &=&\frac{4}{p-1}\left(\frac{a-2\alpha}{2\alpha}\right)^{p-1}.
\end{eqnarray*}

In particular, we may choose 
\[
 a\ge 2\alpha\left[\left(\frac{p-1}{8CC_1}\right)^{1/(p-1)} +1\right],
\]
so that  
\begin{eqnarray}\label{defa}
 C_1\int_{\R\setminus  \tilde{I}_{n}} \frac{\Im \lambda_{n}^{p-1}}
  {|x-\lambda_{n}|^p} \d x\le \frac{1}{2C}.
\end{eqnarray}
Without loss of generality we can assume that $a$ is an integer.
From \eqref{eq2} and \eqref{sameq} we deduce that
\begin{equation}\label{lowereq}
  \int_{\Gamma\cap I_{n}^{a}} 
 \frac{y_{n}^{p-1}}{|x-\lambda_{n}|^{p}} \d x\ge \frac{1}{2CC_1}.
\end{equation}

\smallskip 

We will show that
there exists $\gamma_2>0$ such that
for every $n$, there is $\ell\le a$, with  
\begin{equation}\label{cond2}
 |\Gamma\cap (I_n^{\ell+1}\setminus I_n^{\ell})|\ge\gamma_2 |I_n|.
\end{equation}
Once this is established we can deduce condition \eqref{denscond}.
Indeed, this follows from the following estimate:
\[
 |\Gamma\cap I_n^a|\ge|\Gamma\cap I_n^{\ell+1}\setminus I_n^{\ell}|\ge\gamma_2 |I_n|
 =\frac{\gamma_2}{a}a|I_n|,
\]
which yields \eqref{denscond} with $\gamma=\gamma_2/a$.

\smallskip

It thus remains to prove \eqref{cond2}.
Suppose to the contrary that
for every $\delta>0$, there is $n_{\delta}$ such that for every $\ell\le a-1$ we have
\begin{eqnarray}\label{condneg}
 |\Gamma\cap (I_{n_{\delta}}^{\ell+1}\setminus I_{n_{\delta}}^{\ell})|\le\delta |I_{n_{\delta}}|.
\end{eqnarray}
Recall that $|\lambda_{n_\delta}-t_{n_{\delta}}|
\le \alpha |I_{n_{\delta}}|$. 

Putting $I_n^0=\emptyset$,
we will now discuss the first integral in \eqref{eq2}. 
\begin{eqnarray*}
\int_{\Gamma\cap I_{n_{\delta}}^a} \frac{y_{n_{\delta}}^{p-1}}{|x-\lambda_{n_{\delta}}|^p} \d x
 &=&\sum_{\ell=0}^{a-1}\int_{\Gamma\cap (I_{n_{\delta}}^{\ell+1}\setminus I_{n_{\delta}}^{\ell})}
  \frac{y_{n_{\delta}}^{p-1}}{|x-\lambda_{n_{\delta}}|^p} \d x\\
 &\le& 2\sum_{\ell=0}^{a-1}\int_{\Gamma\cap (I_{n_{\delta}}^{\ell+1}\setminus I_{n_{\delta}}^{\ell})}
  \frac{y_{n_{\delta}}^{p-1}}{(x-x_{n_{\delta}})^p+y_{n_{\delta}}^p} \d x.
\end{eqnarray*}
Let now $u_{\ell}$ be the closest point of $I_{n_{\delta}}^{\ell+1}\setminus I_{n_{\delta}}^{\ell}$ to
$x_{n_{\delta}}$. Then
\begin{eqnarray*}
\int_{\Gamma\cap I_{n_{\delta}}^a} \frac{y_{n_{\delta}}^{p-1}}{|x-\lambda_{n_{\delta}}|^p} \d x
 &\le& 2\sum_{\ell=0}^{a-1}
 \frac{y_{n_{\delta}}^{p-1}}{(u_{\ell}-x_{n_{\delta}})^p+y_{n_{\delta}}^p}  
 \times |\Gamma\cap (I_{n_{\delta}}^{\ell+1}\setminus I_{n_{\delta}}^{\ell})|\\
 &\le& 2 \delta |I_{n_{\delta}}|\sum_{\ell=0}^{a-1}
 \frac{y_{n_{\delta}}^{p-1}}{(u_{\ell}-x_{n_{\delta}})^p+y_{n_{\delta}}^p} ,
\end{eqnarray*}
where we have used \eqref{condneg}.
Now, $|u_{\ell}-x_{n_{\delta}}|$ essentially behaves as $|\ell-\ell_0|\times |I_n|/2$ where 
$\ell_0$ is such that
$x_{n_{\delta}}\in I_{n_{\delta}}^{\ell_0+1}\setminus I_{n_{\delta}}^{\ell_0}$.
Then we can control the sum by twice the  sum starting in 0.
Also,
keeping in mind that $|I_n|/\alpha\le y_n\le \alpha |I_n|$,
\begin{eqnarray*}
\int_{\Gamma\cap I_{n_{\delta}}^a} \frac{y_{n_{\delta}}^{p-1}}{|x-\lambda_{n_{\delta}}|^p} \d x
 &\le& 4\delta |I_{n_{\delta}}|\sum_{\ell=0}^{a-1}  \frac{y_{n_{\delta}}^{p-1}}{(|I_{n_{\delta}}|\ell/2)^p+y_{n_{\delta}}^p}\\
 &\le&4 \delta |I_{n_{\delta}}| \sum_{\ell=0}^{a-1}  \frac{y_{n_{\delta}}^{p-1}}{|I_{n_{\delta}}|^p}
 \frac{1}{(\ell/2)^p+(y_{n_{\delta}}/|I_{n_{\delta}}|)^p}\\
 &\le&4\delta \sum_{\ell=0}^{\infty}  \alpha^{p-1}
 \frac{1}{(\ell/2)^p+\alpha^{-p}}
\end{eqnarray*}
since the last sum converges. Choosing $\delta$ small enough, we get
\begin{eqnarray*}
\int_{\Gamma\cap I_{n_{\delta}}^a} \frac{y_{n_{\delta}}^{p-1}}{|x-\lambda_{n_{\delta}}|^p} \d x
<\frac{1}{2CC_1}
\end{eqnarray*}
 contradicting \eqref{lowereq}.
\qed

\section{Background on the Baranov-Bernstein inequality and 
reverse Carleson measures for model spaces}
\label{background}

Recall that we consider 
meromorphic inner functions $\Theta$. Associated with $\Theta$ we will need two
constants. The first one comes from Baranov's result on Bernstein inequalities, and the second
one from a reverse Carleson measure result in $K_{\Theta}^p$.

\subsection{Baranov-Bernstein inequalities for model spaces} In order to state Baranov's result we need some more notation. 
Given $\eps\in (0,1)$ we 
have already introduced the sub-level set $L(\Theta,\eps)=\{z\in\C_+: |\Theta(z)|<\eps\}$.
Also, recall that for $x\in \R$, we had $d_\varepsilon(x)=\dist\bigl(x,L(\Theta,\eps)\bigr)$.

The reproducing kernel was defined in \eqref{repker}.
We now need a generalization of this.
Indeed, there is a formula for the $n$-th derivative ({\it see} \cite[Formula (2.2)]{B2} for 
general $n$ or \cite[Formula (7)]{B1} for $n=1$): for $f\in K^p_\Theta$,
$$
f^{(n)}(z)= {n! }\int_\R f(t)\, \overline{k_z(t)}^{\,n+1}\,\mbox{d}t,
$$
(observe that in this formula, $\overline{k_z(t)}^{n+1}$ is the $(n+1)$-th power of $\overline{k_z(t)}$
and not the $(n+1)$-th derivative).

\begin{theorem}[Baranov]\label{BaraBern}
Let $\Theta$ be a meromorphic inner function.
Suppose that $\varepsilon\in (0,1)$, $1<p<\infty$. Then for every $f\in K^p_\Theta$ and every $n\in\N$,
\begin{equation}
\label{eq:bb0}
 \|f^{(n)} d_\varepsilon^n \|_p 
 \leq  C(\Theta,p,\varepsilon) n! \Big(\frac{4}{\varepsilon}\Big)^{n} \|f\|_p.
\end{equation}
\end{theorem}

The statement given here is a slightly more precise quantitative version
of Baranov's  Bernstein inequality (see \cite[Theorem 1.5]{B2})
and  its proof is largely similar to that of \cite{B1,B2}.
We shall reproduce Baranov's 
argument below in order to get the right dependence on $n$ of the constant.

\begin{proof}  Let $f\in  K^p_\Theta$,  $x\in \R$, and write 
\begin{eqnarray*}
\frac{1}{n!} d_\varepsilon^n(x) f^{(n)}(x)&= &d_\varepsilon^n(x) \int_\R f(t)\,\overline{k_x(t)}^{\,n+1}\,\mbox{d}t\\
&=&I_1f(x)+I_2f(x),
\end{eqnarray*}
where 
\begin{eqnarray*}
I_1f(x)&=& d_\varepsilon^n(x) \int_{|t-x|\geq d_\varepsilon(x)/2 } f(t)\,\overline{k_x(t)}^{\,n+1} \,\mbox{d}t,\\
I_2f(x)&=&d_\varepsilon^n(x) \int_{|t-x|< d_\varepsilon(x)/2 }  f(t)\,\overline{k_x(t)}^{\,n+1} \,\mbox{d}t.
\end{eqnarray*}

Put $h(x)=d_\varepsilon(x)/2$. We have 
$$
|I_1f(x)|\leq {2^{n+1}}d_\varepsilon^n(x) \int_{|t-x|\geq d_\varepsilon(x)/2 } \frac{|f(t)|}{|t-x|^{n+1}}\,\mbox{d}t
\leq 2 \cdot 4^{n} \times\widetilde{I_1}f(x)
$$
where
$$
\widetilde{I_1}f(x)=h(x) \int_{|t-x|\geq h(x) } \frac{|f(t)|}{|t-x|^{2}}\,\mbox{d}t.
$$
According to \cite[Theorem 3.1]{B1} $\widetilde{I_1}$ is a bounded operator from $L^p$ to $L^p$ for $p>1$.
Therefore, there exists $c_1(\Theta,\eps,p)$ such that
\begin{equation}
\label{eq:bb1}
\norm{I_1 f}_{L^p}\leq c_1(\Theta,\eps,p) 4^n\norm{f}_{L^p}.
\end{equation}

\smallskip

Let us now estimate $I_2 f$. Since $\Theta$ is a meromorphic inner function, $\Theta$ admits an analytic continuation across $\R$.

By the Schwarz Reflection Principle, 
$$
\widetilde{\Theta}(\zeta)=\left\{
\begin{array}{lll}
\Theta (\zeta)& \text{ if } &\Im \zeta\geq 0,\\
&&\\
1/\overline{\Theta(\overline{\zeta})}& \text{ if }& \Im \zeta\leq 0
\end{array} 
\right.
$$
(in the latter case $\Im \zeta$ has to be sufficiently close to 0 to avoid the poles of $\Theta$).
Let $D_x=D\bigl(x,d_\varepsilon(x)/2\bigr)$ the disc of radius $d_\varepsilon(x)/2$ centered at $x$.
We put   
$$
h_x(\zeta):=\left\{
\begin{array}{lll}

\displaystyle \frac{\widetilde{\Theta}(\zeta)-\Theta(x)}{\zeta-x}&\mbox{if }&\zeta\in D_x\setminus\{x\},\\
&&\\
\Theta'(x)&\mbox{if }&\zeta=x. 
\end{array} 
\right.
$$ 
The function $h_x$ is well defined and analytic.
From the Maximum Principle, we deduce that 
$$
|h_x(\zeta)|\leq \sup_{\zeta\in  \partial D_x}|h_x(\zeta) |\leq 2\frac{1+{1}/{\varepsilon}}{d_\varepsilon(x)}\leq \frac{4}{\varepsilon d_\varepsilon(x)}.
$$ 
It follows that 
\begin{eqnarray*}
I_2f(x)&\leq &d_\varepsilon^n(x) \int_{|t-x|< d_\varepsilon(x)/2 }  |f(t) ||h_x(t)|^{n+1}\,\mbox{d}t\\
&\leq & \frac{2}{\varepsilon} \Big(\frac{4}{\varepsilon}\Big)^n 
\frac{2}{d_\varepsilon(x)} \int_{|t-x|< d_\varepsilon(x)/2 }  |f(t) |\,\mbox{d}t
=\frac{2}{\varepsilon} \left(\frac{4}{\varepsilon}\right)^n Mf(x)
\end{eqnarray*}
where $M\ffi(x)=\sup_{r>0}\dst\frac{1}{2r}\int_{x-r}^{x+r}\ffi(t)\,\mbox{d}t$ is the Hardy-Littlewood maximal operator (recall that $x$ is real). 
Since $M$  is bounded from $L^p$ to $L^p$, we obtain a constant $C(p)$ such that
\begin{equation}
\label{eq:bb2}
\norm{I_2 f}_{L^p}\leq \frac{C(p)}{\varepsilon} \left(\frac{4}{\varepsilon}\right)^n\norm{f}_{L^p}.
\end{equation} 
As a result 
$ \|f^{(n)} d_\varepsilon^n \|_p \leq\dst\frac{n!}{2\pi}(\norm{I_1 f}_{L^p}+\norm{I_2 f}_{L^p})$, and
\eqref{eq:bb1}-\eqref{eq:bb2} imply \eqref{eq:bb0}.
\end{proof}

\subsection{Carleson and sampling measures for model spaces} 
The second result which will be important in this paper concerns reverse Carleson measures for
model spaces. Recall  that the Carleson window of an interval $I$ is given by
\[
 S(I)=\{z=x+iy\in \C_+:x\in I,0<y<|I|\}.
\]

We need the following result about reserve Carleson measures for $K_\Theta^p$,  {\it see} \cite{BFGHR} for the case $p=2$ and 
\cite{FHR} for the general case $p>1$ (and which works without requiring the Carleson measure
condition).

\begin{theorem}[Blandign\`eres {\it et al}]\label{BFGHR}
Let $\Theta$ be an inner function and $\eps>0$.
Let $\mu \in M_{+}(\R\cup \C^+)$.
Then there exists an $N_0 = N_0(\Theta, \varepsilon) > 1$ such that 
if
\begin{equation}\label{eq:reversed-condition-first}
\inf_{I}\frac{\mu(S(I))}{|I|}>0,
\end{equation}
where the infimum is taken over all intervals $I \subset \R$ with 
$$
S\left(I^{N_{0}}\right)\cap L(\Theta,\varepsilon)\neq\varnothing,
$$
then,  for every $f\in  K_{\Theta}^p$,

\begin{equation}\label{eq:reversed-equation}
\int_{\R}|f|^p\d x\leq C(\Theta,\mu,\varepsilon)\int_{\R\cup C^+}|f|^p\d\mu.
\end{equation}
\end{theorem}

The smallest possible constant in \eqref{eq:reversed-equation} will be called {\it reverse Carleson
constant} (it correspond exactly to the sampling constant when $\d\mu=\chi_{\Gamma}\d m$, where
$\chi_{\Gamma}$ is the characteristic function of $\Gamma$).
\\

The theorem above will allow us to make the first reduction we mentioned in the introduction. 
Indeed, it will imply that the sampling constants of certain reference sets  
denoted by $F^{a,\sigma}_0$ in the theorem below are uniform. \\

From now on, we will fix an integer 
\begin{equation}
\label{est:N}
N\ge \max\left((1+\alpha)\sqrt{2}N_0,\frac{40\times 8^{1/p} \alpha }{\varepsilon}\right),
\end{equation}
where $\alpha$ is the Baranov constant from \eqref{BarConst} 
and $N_0$ is given in Theorem \ref{BFGHR}.\\

As discussed earlier, in order to use the Remez-type inequality, we need intervals which
are sufficiently small. For this reason we need to subdivide (uniformly) the Baranov intervals (or their 
amplified companions $I_n^a$).  This will be done now.
Let us start from a fixed Baranov sequence $(I_n)_n$, and partition
\[
 I_n^a=\bigcup_{k=1}^{aN}I^a_{n,k}
\]
where the $I^a_{n,k}$'s are intervals of length $|I^a_{n,k}|=|I^a_n|/(aN)=|I_n|/N$. (With no loss of generality, we may increase $a$ in such a way that $aN$ is an integer).
For an application $\sigma: \N \to\{1,2,\ldots,aN\}$,  set $I_n^{a,\sigma}=
I^a_{n,\sigma(n)}$. Our aim is to compare these intervals with the unamplified 
intervals $I_{k,l}:=I^1_{k,l}$. To this end, for a given $\sigma$, we define 
\[
 A_n^{\sigma}=\{(k,l):I_{k,l}\cap I^{a,\sigma}_n\neq \emptyset\},
\]
so that $A_n^{\sigma}$ is the smallest index set for which
\[
 I_n^{a,\sigma}\subset \bigcup_{(k,l)\in A_n^{\sigma}}I_{k,l}.
\]

\begin{figure}
\includegraphics{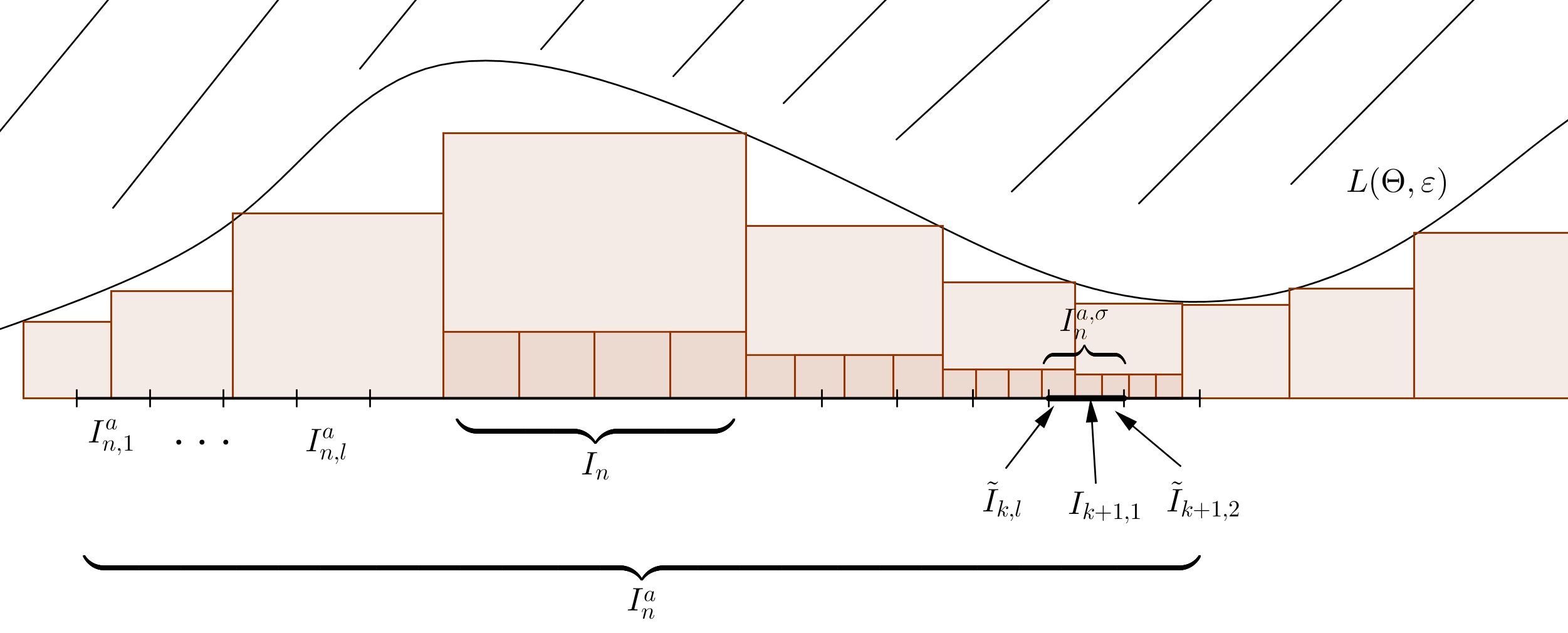}
\caption{Intervals $I_n^a$, $I_{n,k}^a$, $I_n^{a,\sigma}$, $\tilde{I}_{k,l}$}
\end{figure}

We will also use the following notation for $(k,l)\in A_n^{\sigma}$:
\[
 \tilde{I}_{k,l}=
 I_{k,l}\cap I_n^{a,\sigma} 
\]
so that $I_n^{a,\sigma}=\bigcup_{(k,l)\in A_n^{\sigma}}\tilde{I}_{k,l}$.
Then consider the set
\[
 F^{a,\sigma}=\bigcup_{n}
 \left(\bigcup_{(k,l)\in A_n^{\sigma}}\tilde{I}_{k,l}
 \right).
\]

\begin{theorem}\label{thm1}
Let $\Theta$ be an inner function and $1<p<\infty$. Let $\eps>0$, $N_0=N_0(\Theta,\eps)$ as in Theorem \ref{BFGHR}, and let $N$ be as in \eqref{est:N}. 
Suppose $\sigma:\N\to \{1,\ldots,aN\}$.

Given $0<\eta<1$. 
If for every $n$ we choose $A^0_n\subset A^{\sigma}_n$ in such a way that
\begin{eqnarray}\label{dens}
 |\bigcup_{(k,l)\in A^0_n}\tilde{I}_{k,l}|\ge \eta  |I_n^{a,\sigma}| 
 \end{eqnarray}
then 
\[
  F^{a,\sigma}_0=\bigcup_{n}
 \left(\bigcup_{(k,l)\in A_n^{0}}\tilde{I}_{k,l}
 \right)
\]
is uniformly dominating meaning that
there exists a constant $C=C(\Theta,\alpha, p,\varepsilon)$, 
independant of $\sigma$, $\eta $, $a$ and the choice of $A_n^0$ with
\eqref{dens}, such that 
for every $f\in K^p_{\Theta}$, we have
\begin{eqnarray}\label{normequiv}
\int_{F^{a,\sigma}_0}|f(t)|^p\,\mathrm{d}t\le \|f\|_{L^p(\R)}^p\le 
 { e^{C \frac{a^2}{\eta }}}
 \int_{F^{a,\sigma}_0}|f(t)|^p\,\mathrm{d}t.
\end{eqnarray}
\end{theorem}

Before proving this theorem, we discuss the special case $a=1$ that we need for
Corollary \ref{Cora=1}. 
In this case, given $\sigma$, we have $I_n^{a,\sigma}=I_{n,\sigma(n)}$, and hence
$A_n^{\sigma}=\{(n,\sigma(n))\}$. We also choose $A_n^0=A_n^{\sigma}$, so that in \eqref{dens}
we have $\eta=1$. Also $F_0^{\sigma}:=F_0^{1,\sigma}
=\bigcup_{n}I_{n,{\sigma(n)}}$ and we get \eqref{normequiv}
with $e^{C\frac{a^2}{\eta}}$ replaced by $e^C$. Let us state this as a separated result.

\begin{corollary}\label{casea=1}
Let $\Theta$ be an inner function and $1<p<\infty$. Let $\eps>0$, $N_0=N_0(\Theta,\eps)$ as in Theorem \ref{BFGHR}, and let $N$ be as in \eqref{est:N}. 
Suppose $\sigma:\N\to \{1,\ldots,N\}$.

Then the set 
\[
  F^{\sigma}=\bigcup_{n} I_{n,\sigma(n)}
\]
is uniformly dominating meaning that
there exists a constant $C=C(\Theta,\alpha, p,\varepsilon)$, 
independant of $\sigma$, such that 
for every $f\in K^p_{\Theta}$, we have
\begin{eqnarray}\label{normequiv1}
\int_{F^{\sigma}_0}|f(t)|^p\,\mathrm{d}t\le \|f\|_{L^p(\R)}^p\le 
 { e^{C} }
 \int_{F^{\sigma}_0}|f(t)|^p\,\mathrm{d}t.
\end{eqnarray}
\end{corollary}

\begin{figure}
\includegraphics{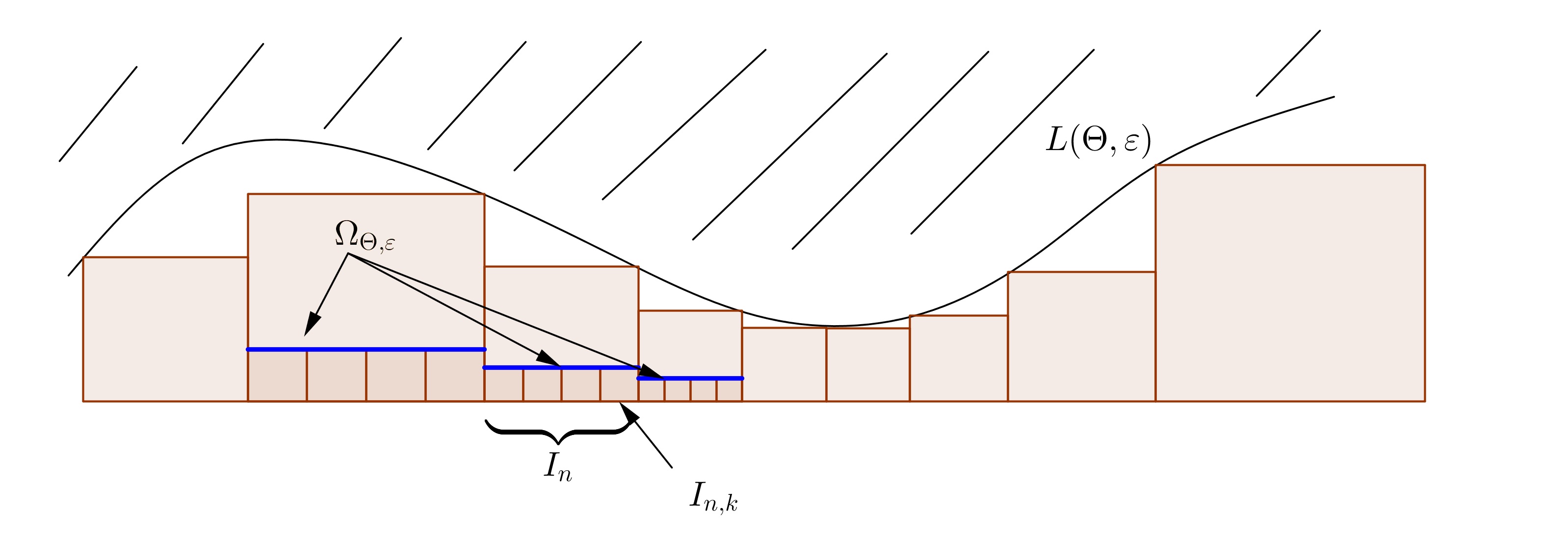}
\caption{Level sets, Baranov intervals, and $\Omega_{\Theta,\eps}$}
\end{figure}

Let us introduce the following notations (see Figure 2). 
Let  
$$
 \Omega_{\Theta,\eps}:=\bigcup_{n\in\Z}\{z=x+iy\in \C_+:x\in I_n,y=|I_n|/N,n\in\N\}.
$$
Consider the measure
\[
\mbox{d}\mu=\sum_{n\in\Z} \mathbf{1}_{I_n}\,\mbox{d}x\otimes\delta_{|I_n|/N}(y),
\]
which defines usual arc length measure on the upper edges of 
$$
\bigcup_{k\in\Z}\bigcup_{l=1,\ldots,N}S(I_{k,l}).
$$

For the proof of Theorem \ref{thm1} we need the following lemmas.

\begin{lemma}\label{lem2}
In the notation above, $\mu$ is a Carleson measure and a reverse Carleson measure.
\end{lemma}

\begin{proof}
Since $\mu$ is the Lebesgue line measure supported on horizontal segments which, projected along
the imaginary axis onto the real line, have intersection of Lebesgue measure zero, it is clearly a
Carleson measure.

Let us now consider the reverse Carleson measure condition. Suppose $I$ is a real interval
with $S(I^{N_0})\cap L(\Theta,\eps)\neq \emptyset$, then for every $x\in I$, and with
\eqref{est:N} in mind, 
\[
 d_{\eps}(x)\le \sqrt{2}N_0|I|\le \frac{N|I|}{1+\alpha}, 
\]
and hence, 
\[
 |I|\ge \frac{ {(1+\alpha)}d_{\eps}(x)}{N},
\]
where $x$ is arbitrary in $I$. On the other hand, there is $n$ such that $x\in I_n$, so that
when $z=x+iy \in \Omega(\Theta,\eps)$ we have $y=|I_n|/N$ and from  
\eqref{BarConst} we know that since $x\in I_n$ we thus get
\[
 y=\frac{|I_n|}{N}\le \frac{\alpha d_{\eps}(x)}{N} \le
 |I|.
\]
As a result, for every $x\in I$, the corresponding $z=x+iy\in \Omega_{\Theta,\eps}$ is in $S(I)$ so that 
$\mu\bigl(S(I)\bigr)=|I|$. Whence \eqref{eq:reversed-condition-first} is fulfilled and we conclude from
Theorem \ref{BFGHR} that $\mu$ is a reverse Carleson measure.
\end{proof}

\begin{lemma}\label{harmmeas}
In the notation above we have
\[
 \inf\{\omega_z(F^{a,\sigma}_0):z\in \Omega_{\Theta,\eps}\}
 \ge \delta=\frac{4\eta   }{\bigl(2+N(a+1)\bigr)^2\pi}>0,
\]
independently of the choice of $\sigma$ and $A_n^0\subset A_n^{\sigma}$ satisfying
\eqref{dens}.
\end{lemma}
 
\begin{proof}

Observe that if $z=x+iy\in \Omega_{\Theta,\eps}$ then there exists $n\in \Z$, 
such that $x\in I_{n}$ and $y=|I_n|/N=|I_n^{a,\sigma}|$. 

Now, for $t\in I_n^{a,\sigma}\subset I_n^a$, the distance from $t$ to $x$ is bounded by the distance of one edge of $I_n$ to the oposit edge of $I_n^a$, that is $|x-t|\leq |I_n|/2+|I_n^a|/2$. Therefore
$$
|z-t|\leq |y|+|x-t|\leq |I_n^{a,\sigma}|+\frac{1}{2}|I_n|+\frac{aN}{2}|I_n^{a,\sigma}|=
|I_n^{a,\sigma}|\left(1+\frac{N(a+1)}{2}\right).$$ Hence
\begin{eqnarray*}
\omega_z(F^{a,\sigma}_0)
&\geq& 
 \omega_z(\bigcup_{(k,l)\in A_n^{0}}\tilde{I}_{k,l})=\frac{1}{\pi}\int_{\bigcup_{(k,l)\in A_n^{0}}
 \tilde{I}_{k,l}}\frac{|I_n^{a,\sigma}|}{|z-t|^2}\,\mbox{d}t\\
&\geq&\frac{4\times |\bigcup_{(k,l)\in A_n^{0}}\tilde{I}_{k,l}|
 \times |I_n^{a,\sigma}|}{\bigl(2+N(a+1)\bigr)^2\pi|I_n^{a,\sigma}|^2}\\
&\ge& \frac{4\eta }{\bigl(2+N(a+1)\bigr)^2\pi}
\end{eqnarray*}
which proves the lemma. 
\end{proof}

\begin{proof}[Proof of Theorem \ref{thm1}]
The left hand inequality (Carleson embedding) is immediate.\\

Let us consider the right hand embedding (reverse Carleson inequality).
From an idea of Havin-J\"oricke \cite{HJ} and Dyakonov \cite{Dy}, we know that for every $1<q<+\infty$, and for every $f\in 
H^p$, we have the Jensen inequality 
\begin{eqnarray}
 |f(z)|^q&\le& 2\left(\int_{F^{a,\sigma}_0} |f(t)|^qP_z(t)\,\mbox{d}t\right)^{\omega_z(F^{a,\sigma}_0)}
 \left(\int_{\R} |f(t)|^qP_z(t)\,\mbox{d}t\right)^{1-\omega_z(F^{a,\sigma}_0)}\nonumber\\
&=&2\int_{\R} |f(t)|^qP_z(t)\,\mbox{d}t\left(\frac{\dst\int_{F^{a,\sigma}_0} |f(t)|^qP_z(t)\,\mbox{d}t}
{\dst\int_{\R} |f(t)|^qP_z(t)\,\mbox{d}t}\right)^{\omega_z(F^{a,\sigma}_0)}\label{Dyak1}
\end{eqnarray}
where $P_z$ is the  Poisson kernel in the upper half place.
Recall that
\[
 \Omega_{\Theta,\eps}=\bigcup_{n\in\Z}\{z=x+iy\in \C_+:x\in I_n,y=|I_n|/N,n\in\N\}.
\]
It follows from Lemma \ref{harmmeas} and \eqref{Dyak1} that, for $z\in \Omega_{\Theta,\eps}$,
\begin{eqnarray}
 |f(z)|^q&\le&2\int_{\R} |f(t)|^qP_z(t)\,\mbox{d}t\left(\frac{\dst\int_{F^{a,\sigma}_0} |f(t)|^qP_z(t)\,\mbox{d}t}
{\dst\int_{\R} |f(t)|^qP_z(t)\,\mbox{d}t}\right)^\delta\nonumber\\
&=& 2\left(\int_{F^{a,\sigma}_0} |f(t)|^qP_z(t)\,\mbox{d}t\right)^{\delta}
\left(\int_{\R} |f(t)|^qP_z(t)\,\mbox{d}t\right)^{1-\delta}\label{Dyak2}
\end{eqnarray}

Now pick a function $f\in L^p(\R)$ and let $s=\dst\frac{1}{2}\left(1-\frac{1}{p}\right)$ so that $(1-s)p=\frac{1+p}{2}$.
Note that $0<s<1$ and that $1<(1-s)p<p$.

Write $q=(1-s)p$ and define the two harmonic functions $u(z)=\dst\int_\R |f(t)|^qP_z(t)\,\mbox{d}t$
and $u_\sigma(z)=\dst\int_\R\chi_{F^{a,\sigma}_0}(t)|f(t)|^qP_z(t)\,\mbox{d}t$. Then
\eqref{Dyak2} reads as
\begin{equation}
|f(z)|^q\leq 2 u_\sigma(z)^\delta u(z)^{1-\delta}.
\label{eq:Dyak2bis}
\end{equation}
Since by Lemma \ref{lem2} 
$\mu$ is a Carleson measure, and so, in view of \cite[Theorem I.5.6]{Gar}, 
there exists a constant $C(p)$ such that, for every $\ffi\in L^{1/(1-s)}(\R)=L^{\frac{1+p}{2p}}(\R)$,
$$
\int_{\Omega_{\Theta,\eps}}\ent{\int_{\R} |\ffi(t)|P_z(t)\,\mbox{d}t}^{\frac{1}{1-s}}\,\mbox{d}\mu(z)
\leq C(p)\int_\R |\ffi(t)|^{\frac{1}{1-s}}\,\mbox{d}t.
$$
Applying this to $\ffi=\chi_{F_0^{a,\sigma}}|f|^{(1-s)p}$ and to $\ffi=|f|^{(1-s)p}$, which 
are both in $L^{1/(1-s)}$, we get
\begin{eqnarray}
\int_{\Omega_{\Theta,\eps}}u_\sigma(z)^{\frac{1}{1-s}}\,\mbox{d}\mu(z)&\leq&
C(p)\int_{F^{a,\sigma}_0} |f(t)|^p\,\mbox{d}t\label{eq:Carleson1}\\
\int_{\Omega_{\Theta,\eps}}u(z)^{\frac{1}{1-s}}\,\mbox{d}\mu(z)&\leq&
C(p)\int_{\R} |f(t)|^p\,\mbox{d}t.\label{eq:Carleson2}
\end{eqnarray}

Now, integrating \eqref{eq:Dyak2bis} with respect to $\mu$ we get with \eqref{Dyak2}
\begin{eqnarray*}
\int_{\Omega_{\Theta,\eps}}|f(z)|^p\,\mbox{d}\mu(z)
&=&\int_{\Omega_{\Theta,\eps}}|f(z)|^{s p}|f(z)|^q\,\mbox{d}\mu(z)\\
&\leq& 2\int_{\Omega_{\Theta,\eps}}|f(z)|^{s p}u_\sigma(z)^\delta u(z)^{1-\delta}\,\mbox{d}\mu(z)\\
&\leq&2\left(\int_{\Omega_{\Theta,\eps}}|f(z)|^p\,\mbox{d}\mu(z)\right)^s
\left(\int_{\Omega_{\Theta,\eps}}u_\sigma(z)^{\frac{\delta}{1-s}}u(z)^{\frac{1-\delta}{1-s}}\,\mbox{d}\mu(z)\right)^{1-s}
\end{eqnarray*}
where we have applied
H\"older's inequality with exponents $1/s$, $1/(1-s)$. It follows that
\begin{eqnarray*}
\int_{\Omega_{\Theta,\eps}}|f(z)|^p\,\mbox{d}\mu(z)&\leq& 2^{\frac{1}{1-s}}\int_{\Omega_{\Theta,\eps}}u_\sigma(z)^{\frac{\delta}{1-s}}u(z)^{\frac{1-\delta}{1-s}}\,\mbox{d}\mu(z)\\
&\leq& 2^{\frac{2p}{1+p}}\left(\int_{\Omega_{\Theta,\eps}}u_\sigma(z)^{\frac{1}{1-s}}\,\mbox{d}\mu(z)\right)^\delta
\left(\int_{\Omega_{\Theta,\eps}}u(z)^{\frac{1}{1-s}}\,\mbox{d}\mu(z)\right)^{1-\delta}
\end{eqnarray*}
where we have again used H\"older's inequality, now with exponents $1/\delta$, $1/(1-\delta)$
($\delta$ can be assumed in $(0,1)$).
Using \eqref{eq:Carleson1}-\eqref{eq:Carleson2} this gives
$$
\int_{\Omega_{\Theta,\eps}}|f(z)|^p\,\mbox{d}\mu(z)\leq
2^{\frac{2p}{1+p}}C(p)\left(\int_{F^{a,\sigma}_0}|f(t)|^p\,\mbox{d}t\right)^\delta
\left(\int_{\R}|f(t)|^p\,\mbox{d}t\right)^{1-\delta}.
$$

On the other hand, from Lemma \ref{lem2}, we know that $\mu$ is reverse Carleson with constant $C_R$, so
\[
 \int_\R|f(t)|^p\,\mbox{d}t\le C_R^p \int_{\Omega_{\Theta,\eps}}|f|^p\,\mbox{d}\mu
 \le C_R2^{\frac{2p}{1+p}}C(p)\left(\int_{F^{a,\sigma}_0} |f(t)|^p\,\mbox{d}t\right)^{\delta}
 \left(\int_{\R} |f(t)|^p\,\mbox{d}t\right)^{1-\delta},
\]
which yields
\[
 \int_{\R}|f(t)|^p\,\mbox{d}t\le \Big(C_R2^{\frac{2p}{1+p}}C(p)\Big)^{1/\delta}\int_{F^{a,\sigma}_0}|f(t)|^p\,\mbox{d}t.
\]
It remains to remember the form of $\delta$ as given in Lemma \eqref{harmmeas}
to conclude.
\end{proof}

\section{Proof of (i) implies (iii) in Theorem \ref{thm2} and estimate of the constants.} 

In view of the construction of $F^{a,\sigma}_0$ the main idea is to switch to the sets
$\tilde{I}_{k,l}$, $(k,l)\in A_m^0$ ($k,l,m$ appropriate).

Suppose the set $\Gamma$ is $(\gamma,a)$-relatively dense with respect to the Baranov sequence
$(I_n)_n$:
\[
 |\Gamma\cap I_n^a|\ge \gamma a |I_n|=\gamma |I_n^a|.
\]
Since the $I_{n,k}^a$'s partition $I_n^a$, this implies  that for every $n$ there exists at least
one $k$, denoted by $k=\sigma(n)$, such that
\[
  |\Gamma\cap I_{n,k}^{a}|\ge \gamma  |I_{n,k}^{a}|.
\]
By our previously introduced notation $I_n^{a,\sigma}=I^a_{n,\sigma(n)}$.
Recall that $A_n^{\sigma}=\{(k,l):I_{k,l}\cap I_n^{a,\sigma}\neq\emptyset\}$ and $I_n^{a,\sigma}=
\bigcup_{(k,l)\in A_n^{\sigma}}\tilde{I}_{k,l}$.
Then the relative density condition translates to
\begin{eqnarray}\label{eq5}
\sum_{(k,l)\in A_n^{\sigma}}|\Gamma\cap\tilde{I}_{k,l}|
 =|\Gamma\cap I_n^{a,\sigma}|\ge\gamma |I_n^{a,\sigma}|.
\end{eqnarray}
Set
\[
 A_n^{0}=\{(k,l)\in A_n^{\sigma}:|\Gamma\cap\tilde{I}_{k,l}|\ge \frac{\gamma}{2}|\tilde{I}_{k,l}|\}.
\]
then decomposing $A_n^{\sigma}=A_n^0\cup A_n^{\sigma}\setminus A_n^0$, we deduce from
\eqref{eq5}
\[
 \sum_{(k,l)\in A_n^{0}}|\Gamma\cap\tilde{I}_{k,l}|
 \ge \gamma |I_n^{a,\sigma}|-\frac{\gamma}{2}\sum_{(k,l)\in A_n^\sigma\setminus A_n^0}|\tilde{I}_{k,l}|
 \ge \frac{\gamma}{2}|I_n^{a,\sigma}|,
\]
which in particular yields condition \eqref{dens} with 
$$\eta =\gamma/2.$$

\smallskip

From now on we will use the notation
\[
 F:=F^{a,\sigma}_0=\bigcup_n J_n,
\]
where, as above, $J_n=\tilde{I}_{k,l}$ for an appropriate $(k,l)\in A_n^0$. 
In particular, we deduce from the very definition of $A_n^0$ that for every $n$
\[
 |\Gamma\cap J_n|\ge \frac{\gamma}{2}|J_n|. 
\]
We should also recall that since $J_n=\tilde{I}_{k,l}\subset I_{k,l}$ we have
for every $x\in J_n$, $|J_n|\le \frac{\alpha}{N} \dist(x,L(\Theta,\eps))$.
\medskip

A main ingredient of our proof is a Remez-type inequality which requires the control of the uniform norm
on a bigger set depending on that of a smaller set. Our method only works when the sets
are sufficiently small. For that reason we have to reduce the situation to sufficiently small sets which
will be achieved using Theorem \ref{thm1}.

\subsection{Step 1  --- Reduction to the dominating set $F$}

\begin{corollary}\label{cor:localbernstein}
With the notation of Theorems \ref{BaraBern} and \ref{thm1}, there exists a constant $\widetilde{C}=C(\Theta,\alpha,p,\varepsilon)$, depending only on $\Theta, \alpha, p$ and $\varepsilon$
such that, for every $f\in K^p_{\Theta}$ and every $k\in\N$, we have
\begin{equation}
\int_{F}\left({|f^{(k)}(x)|}{d_\varepsilon(x)^k}\right)^p\d x
\leq e^{\widetilde{C} \frac{a^2 }{\gamma }} \left(\frac{4^kk!}{\varepsilon^k}\right)^p \int_{F}|f(x)|^p\d x.
\label{condLS1bis}
\end{equation}
\end{corollary}

\begin{proof} Indeed, using successively a trivial estimate, Theorem \ref{BaraBern}
and Theorem \ref{thm1}, we get for every $f\in K^p_{\Theta}$,
\begin{eqnarray*}
\int_{F}\left({|f^{(k)}(x)|}{d_\varepsilon(x)^k}\right)^p\,\mbox{d}x
 &\le& \int_{\R}\left({|f^{(k)}(x)|}{d_\varepsilon(x)^k}\right)^p\,\mbox{d}x\\
 &\le& C^p(\Theta,p,\eps) 
   \left(\frac{4^kk!}{\varepsilon^k}\right)^p\int_{\R}|f(x)|^p\,\mbox{d}x \\
 &\le&   C^p({\Theta,p,\varepsilon} )\left(\frac{4^kk!}{\varepsilon^k}\right)^p e^{C\frac{a^2 }{\gamma }}\int_{F}|f(x)|^p\,\mbox{d}x
\end{eqnarray*}
as claimed.
\end{proof}

Inequality \eqref{condLS1bis} means that we have a Bernstein inequality with respect to 
$F$ so that we can
replace $\R$ by $F$ in Theorem \ref{thm2}.

\subsection{Step 2 --- Good intervals}

Assume that \eqref{condLS1bis} holds. 
An integer $n$ and the corresponding interval will be called \emph{bad} if there exists
an integer $m_n$ such that
$$
\int_{J_n}\left({|f^{(m_n)}(x)|}{d_\varepsilon(x)^{m_n}}\right)^p\,\mbox{d}x\geq 
  {e^{\widetilde{C} \frac{a^2 }{\gamma }}} 4^{m_n}   \Big(\frac{4^{m_n}m_n!}{\varepsilon^{m_n}}\Big)^p\int_{J_n}|f(x)|^p\,\mbox{d}x.
$$
Observe that $m_n\ge 1$.
We will say that $n$ and $J_n$ are {\em good} if they are not bad.\\

\medskip

\noindent{\bf Claim 1.} {\sl The good intervals contain most of the mass of $f$ in the sense that}
$$
\int_{\cup_{n \; is \; good}J_n}
 \abs{ f(x)}^p \mbox{d}x\geq \frac{2}{3}\int_{F}\abs{ f(x)}^p \mbox{d}x.
$$

\medskip

\begin{proof}[Proof of Claim 1] By definition of bad intervals
\begin{eqnarray*}
\int_{\cup_{n \; is \; bad}J_n}\abs{ f(x)}^p \mbox{d}x
&=&\sum_{n \; is \; bad}\int_{J_n}\abs{ f(x)}^p \mbox{d}x\\
&\leq &\sum_{n \; is \; bad}\frac{{e^{-\widetilde{C} \frac{a^2 }{\gamma }}}}{  4^{m_n}}\Big(\frac{\varepsilon^{m_n}}{4^{m_n} m_n!}\Big)^p\int_{J_n}
 \left({|f^{(m_n)}(x)|}{d_\varepsilon(x)^{m_n}}\right)^p\,\mbox{d}x.
\end{eqnarray*}
Now
\begin{multline*}
\frac{{e^{-\widetilde{C} \frac{a^2 }{\gamma }}}}{  4^{m_n}}\Big(\frac{\varepsilon^{m_n}}{4^{m_n} m_n!}\Big)^p \int_{J_n}
 \left({|f^{(m_n)}(x)|}{d_\varepsilon (x)^{m_n}}\right)^p
 \,\mbox{d}x\\
 \leq \sum_{k\ge 1}  \frac{{e^{-\widetilde{C} \frac{a^2 }{\gamma }}}}{ 4^{k}}\Big(\frac{\varepsilon^{k}}{4^k k!}\Big)^p\int_{J_n}
 \left({|f^{(k)}(x)|}{d_\varepsilon (x)^{k}}\right)^p\,\mbox{d}x.
\end{multline*}
By Fubini's theorem we get 
\begin{eqnarray*}
\int_{\cup_{n \; is \; bad}J_n}\abs{ f(x)}^p\,\mbox{d}x
&\leq & \sum_{k\ge 1}  \frac{{e^{-\widetilde{C} \frac{a^2 }{\gamma }}}}{ 4^{k}}\Big(\frac{\varepsilon^{k}}{4^k k!}\Big)^p\int_{\dst\cup_{n \; is \; bad}J_n}
  \left({|f^{(k)}(x)|}{d_\varepsilon(x)^{k}}\right)^p\,\mbox{d}x\\
&\leq &  \sum_{k\ge 1}  \frac{{e^{-\widetilde{C} \frac{a^2 }{\gamma }}}}{  4^{k}}\Big(\frac{\varepsilon^{k}}{4^k k!}\Big)^p\int_{F}
  \left({|f^{(k)}(x)|}{d_\varepsilon(x)^{k}}\right)^p \mbox{d}x.
\end{eqnarray*}
Bernstein's Inequality \eqref{condLS1bis} then implies
$$
\int_{\cup_{n \; is \; bad}J_n}\abs{ f(x)}^p \mbox{d}x
\leq\sum _{k\geq 1} \frac{1}{4^{k}}  \int_{F}\abs{ f(x)}^p \mbox{d}x
= \frac{1}{3}\int_{F}\abs{ f(x)}^p \mbox{d}x,
$$
from which Claim 1 follows.
\end{proof}

\medskip

\subsection{Step 3 --- Good points}

Let $\kappa>1$.
For each {\em good} $n$, we will say that a point $x\in J_n$ is {\em $\kappa$-good} if, for every $k\in\N$,
$$
|f^{(k)}(x)|^p\leq 2\kappa  {e^{\widetilde{C} \frac{a^2 }{\gamma }}} 
 \times 8^k\Big(\frac{ 4^{k} k!}{\varepsilon ^kd_\varepsilon(x)^k}\Big)^p\frac{1}{|J_n|}\int_{J_n}|f(x)|^p\,\mbox{d}x.
$$

\noindent{\bf Claim 2.} {\sl 
Let $G_n$ be the set of $\kappa$-good points in $J_n$. Then $|G_n|\geq\dst\left(1-\frac{1}{\kappa}\right)|J_n|$.}

\begin{remark}
Let $\kappa=\frac{4}{\gamma}$. 
As  $|J_n\cap\Gamma|\geq\frac{\gamma}{2} |J_n|$, we get
$$
|G_n\cap\Gamma|\geq |J_n\cap\Gamma|-|J_n\setminus G_n|\geq \frac{\gamma}{2} |J_n|-\frac{\gamma}
 {4}|J_n|
=\frac{\gamma}{4} |J_n|\geq\frac{\gamma}{4} |G_n|.
$$
This means that there are many good points in $\Gamma$, but we shall not use this fact.
\end{remark}

\medskip

\begin{proof}[Proof of Claim 2] 
Let $B_n=J_n\setminus G_n$ be the set of bad points. Then for every
$x \in B_n$, there exists $ k_x \geq 0$ such that
$$
\frac{1}{|J_n|}\int_{J_n} \abs{ f(y)}^p \mbox{d}y \leq \frac{{e^{-\widetilde{C} \frac{a^2 }{\gamma }}}}{2\kappa  \times 8^{k_x}}
\Big( \frac{\varepsilon^{k_x} d_\varepsilon(x)^{k_x}}{4^{k_x} k_x!}\Big)^p\abs{f^{(k_x)}(x)}^p.
$$
Therefore
$$
\frac{1}{|J_n|}\int_{J_n} \abs{ f(y)}^p \mbox{d}y \leq \sum_{k\geq0}
   \frac{{e^{-\widetilde{C} \frac{a^2 }{\gamma }}}}{2\kappa  \times 8^{k}}
\Big( \frac{\varepsilon^k d_\varepsilon(x)^{k}}{4^k k!}\Big)^p
\abs{ f^{(k)}(x)}^p.
$$
Integrating both sides over $B_n$, we obtain
\begin{eqnarray*}
\frac{|B_n|}{|J_n|}\int_{J_n} \abs{ f(y)}^p\,\mbox{d}y &\leq& \sum_{k\geq0}
\frac{{e^{-\widetilde{C} \frac{a^2 }{\gamma }}}}{2\kappa   \times 8^{k}}\Big(\frac{\varepsilon^k}{4^{k} k!}\Big)^p
\int_{B_n}\Big(| f^{(k)}(x)|d_\varepsilon(x)^{k}\Big)^p\,\mbox{d}x\\
&\leq& \sum_{k\geq0}
\frac{{e^{-\widetilde{C} \frac{a^2 }{\gamma }}}}{2\kappa\times  8^{k}}\Big(\frac{\varepsilon^k}{4^{k} k!}\Big)^p
\int_{J_n}\Big(| f^{(k)}(x)|d_\varepsilon(x)^{k}\Big)^p\,\mbox{d}x.
\end{eqnarray*}
Now, since $n$ is good, we get
$$
\frac{|B_n|}{|J_n|}\int_{J_n} \abs{ f(y)}^p \mbox{d}y \leq \sum_{k\geq0}\frac{4^{k}}{2\kappa \times 8^{k}}
\int_{J_n}|f(x)|^p\,\mbox{d}x
=\frac{1}{\kappa }\int_{J_n}|f(x)|^p\,\mbox{d}x.
$$
Since $f$ cannot be $0$ almost everywhere on $J_n$, $\dst\int_{J_n}|f(x)|^p\,\mbox{d}x\not=0$
and we get $|B_n|\leq\dst\frac{1}{\kappa}|J_n|$ which yields Claim 2.
\end{proof}

\medskip

In the next step, we will need a Remez type inequality. There exist different 
versions of such inequalities, {\it e.g.}
\cite[Lemma B]{NSV}. The one that seems most suitable for our needs 
is the following  straightforward adaptation of a
result of O. Kovrijkine \cite[Corollary, p 3041]{Ko} {\it see also} \cite[Theorem 4.3]{GJ}.

\begin{lemma}[Kovrijkine's Remez Type Inequality] \label{lem:ko}
Let $p\in[1,\infty)$.
Let $\Phi$ be an analytic function, $J$ an interval and $E\subset J$ a set of positive
measure.
Let $M=\max_{D_{J}}|\Phi(z)|$ where $D_{J}=\{z\in\C, dist(z,J)<4|J|\}$ and let $m=\max_{J}|\Phi(x)|$, then
\begin{equation}
\label{eq:ko}
\int_J|\Phi(s)|^p\d s\leq\left(\frac{300|J|}{|E|}\right)^{p\frac{\ln (M/m)}{\ln 2}+1}\int_E|\Phi(s)|^p\d s.
\end{equation}
\end{lemma}

\subsection{Step 4 --- Conclusion}

\medskip

It remains to apply Lemma \ref{lem:ko} with $\Phi=f$, $J=J_n$ a good interval, 
$E=\Gamma\cap J_n$. We write $M=\max_{y\in D_{J_n}}|f(y)|$ and $m=\max_{x\in J_n}|f(x)|$.

First, note that if $x\in J_n$ is $\kappa$-good then, by assumption \eqref{est:N} on $N$,
\begin{equation}
\label{eq:jndx}
 \frac{|J_n|}{d_\varepsilon(x)} \le \frac{\alpha }{N}\le \frac{\varepsilon}{40\times 8^{1/p}}.
\end{equation}
Further, for such an $x$, and a $y$ with $|x-y|<10|J_n|$, we get
\begin{eqnarray*}
|f(y)|&\leq&\sum_{k\geq0}\frac{|f^{(k)}(x)|}{k!}|x-y|^k\\
&\leq& (2\kappa  {e^{\widetilde{C} \frac{a^2 }{\gamma }}} )^{1/p}
\sum_{k\geq0} 8^{k/p}\left(\frac{4|x-y|}{\varepsilon d_\varepsilon(x)}\right)^k
\left(\frac{1}{|J_n|}\int_{J_n}|f(x)|^p\,\mbox{d}x\right)^{1/p}\\
&=&(2\kappa  {e^{\widetilde{C} \frac{a^2 }{\gamma }}} )^{1/p}
\sum_{k\geq0} \left(\frac{|x-y|}{10|J_n|}\right)^k\left(\frac{40\times 8^{1/p}|J_n|}{\varepsilon d_\varepsilon(x)}\right)^k
\left(\frac{1}{|J_n|}\int_{J_n}|f(x)|^p\,\mbox{d}x\right)^{1/p}\\
&\leq&(2\kappa  {e^{\widetilde{C} \frac{a^2 }{\gamma }}} )^{1/p}
\sum_{k\geq0} \left(\frac{|x-y|}{10|J_n|}\right)^k
\sup_{J_n}|f(x)|^p
\end{eqnarray*}
with \eqref{eq:jndx}. Summing this last series, we obtain
\begin{equation}
\label{eq:estfy}
|f(y)|\leq\frac{(2\kappa  {e^{\widetilde{C} \frac{a^2 }{\gamma }}} )^{1/p}}{1-\frac{|x-y|}{10|J_n|}}\max_{J_n}|f(x)|.
\end{equation}
But now, from Claim 2, we know that the set $G_n$ of $\kappa$-good points has measure $|G_n|\geq\left(1-\frac{1}{\kappa}\right)|J_n|$.
Thus, if $y\in D_{J_n}$, there exists $z\in J_n$ such that $|z-y|\leq 4|J_n|$, and there exists $x\in G_n$ such that
$|z-x|\leq \dst\frac{1}{2\kappa}|J_n|$ which yields $|y-x|\leq \dst\left(4+\frac{1}{2\kappa}\right)|J_n|$. 
We can now chose $\kappa$ to be such that
$\dst \frac{\kappa^{1/p}}{1-\frac{1}{10}\left(4+\frac{1}{2\kappa}\right)}\leq 2$
(which is possible since the left hand side goes to $20/11<2$ when $\kappa\to 1$).
Then, \eqref{eq:estfy} implies
\begin{eqnarray}\label{maxestimate}
\max_{y\in D_{J_n}}|f(y)|\le 2\times (2  {e^{\widetilde{C} \frac{a^2 }{\gamma }}} )^{1/p}\max_{x\in J_n}|f(x)|,
\end{eqnarray} 
from which we obtain
$M/m\le 2\times (2 {e^{\widetilde{C} \frac{a^2 }{\gamma }}} )^{1/p}$. Since $|\Gamma\cap J_n|\geq\frac{\gamma}{2}|J_n|$, 
Kovrijkine's Remez Type Inequality then reads
\begin{eqnarray}\label{finalestimate}
\int_{J_n}|f(x)|^p\d x
\leq \left(\frac{C_1}{\gamma}  \right)^{C_2\frac{a^2}{\gamma }}\int_{J_n\cap \Gamma}|f(x)|^p\d x
\end{eqnarray}
where $C_1=C_1(\Theta,\alpha,p,\varepsilon)$.  Summing over all good intervals gives the result. 
\qed
\bigskip

We finish this section commenting on the proof of Corollary \ref{Cora=1}. We first observe that in view
of Corollary \ref{casea=1} the constant $e^{\tilde{C}a^2/\gamma}$ appearing in \eqref{condLS1bis}
turns out to be $e^{\tilde{C}}$, with $\tilde{C}>C$ where $C$ is the constant in \eqref{normequiv1}.
With this in mind, and following the lines of the proof above we see that reaching  
\eqref{maxestimate} the constant does not depend on $\gamma$, so that finally the exponent
in \eqref{finalestimate} is just a constant as required \qed

\subsection*{Acknowledgments}
The authors would like to thank Anton Baranov for several fruitful discussions in connection with
the results of this paper.

\end{document}